\documentclass{amsart}
\usepackage{xcolor}
\definecolor{nicered}{rgb}{0.6, 0, 0.1}
\definecolor{niceblue}{rgb}{0.06, 0.3, 0.57}
\definecolor{nicegreen}{rgb}{0.0, 0.51, 0.5}
\usepackage[colorlinks=true,pagebackref,hyperindex,citecolor=nicegreen,linkcolor=nicered,urlcolor=niceblue]{hyperref}
\usepackage{amsmath}
\usepackage{amsfonts}
\usepackage{amssymb}
\usepackage{microtype}
\usepackage{mathrsfs}
\usepackage{mathtools}
\usepackage[mathcal]{euscript}

\usepackage{cleveref}
\crefname{theorem}{Theorem}{Theorems} 
\crefname{lemma}{Lemma}{Lemmas}
\crefname{corollary}{Corollary}{Corollaries}
\crefname{proposition}{Proposition}{Propositions}
\crefname{definition}{Definition}{Definitions}
\crefname{convention}{Convention}{Conventions}
\crefname{remark}{Remark}{Remarks}
\crefname{example}{Example}{Examples}
\crefname{notation}{Notation}{Notations}
\crefname{setup}{Setup}{Setups}

\newtheorem{theorem}{Theorem}[section]
\newtheorem{lemma}[theorem]{Lemma}
\newtheorem{corollary}[theorem]{Corollary}
\newtheorem{proposition}[theorem]{Proposition}

\newtheorem{thmintro}{Theorem}

\theoremstyle{definition}
\newtheorem{definition}[theorem]{Definition}
\newtheorem{convention}[theorem]{Convention}
\newtheorem{notation}[theorem]{Notation}
\newtheorem{example}[theorem]{Example}
\newtheorem{setup}[theorem]{Setup}

\theoremstyle{remark}
\newtheorem{remark}[theorem]{Remark}

\numberwithin{equation}{section} 

\newcommand{\FF}{\mathbb{F}}
\newcommand{\NN}{\mathbb{N}}
\newcommand{\RR}{\mathbb{R}}
\newcommand{\QQ}{\mathbb{Q}}
\newcommand{\kk}{\Bbbk}
\newcommand{\VV}{\mathbb{V}}

\newcommand{\vv}[1]{\mathbf{{#1}}}
\newcommand{\one}{\vv{1}}
\newcommand{\canvec}{\vv{e}}
\newcommand{\norm}[1]{ \| #1 \| }
\newcommand{\scalarvector}[2]{#1 #2}
\newcommand{\vectorscalar}[2]{#2 #1}

\DeclarePairedDelimiter{\ideal}{\langle}{\rangle}
\newcommand{\ideala}{\mathfrak{a}}
\newcommand{\idealb}{\mathfrak{b}}
\newcommand{\ideald}{\mathfrak{d}}
\newcommand{\idealm}{\mathfrak{m}}
\newcommand{\idealn}{\mathfrak{n}}
\newcommand{\frob}[2]{#1^{[#2]}}

\newcommand{\up}[1]{\left\lceil #1 \right\rceil}
\newcommand{\down}[1]{\left\lfloor #1 \right\rfloor}
\newcommand{\lpr}[2]{ [ \hspace{.01in} #1 \, \% \, #2 \hspace{.01in} ]} 

\DeclareMathOperator{\diag}{diag}
\DeclareMathOperator{\comp}{comp}
\DeclareMathOperator{\crit}{crit}

\newcommand{\cf}{cf.}
\newcommand{\eg}{e.g., }

\newcommand{\D}{d}
\renewcommand{\geq}{\geqslant}
\renewcommand{\leq}{\leqslant}

\begin{document}

\title{Frobenius powers of some monomial ideals}
\author{Daniel J.~Hern\'andez}
\address{Department of Mathematics, University of Kansas, Lawrence, KS~66045, USA}
\email{hernandez@ku.edu}

\author{Pedro Teixeira}
\address{Department of Mathematics, Knox College, Galesburg, IL~61401, USA}
\email{pteixeir@knox.edu}

\author{Emily E.~Witt}
\address{Department of Mathematics, University of Kansas, Lawrence, KS~66045, USA}
\email{witt@ku.edu}
\thanks{Partial support was provided by NSF grants DMS-1600702 (first author) and DMS-1623035 (third author).}


\date{}

\begin{abstract}
    In this paper, we characterize the (generalized) Frobenius powers and critical exponents of two classes of monomial ideals of a polynomial ring in positive characteristic:  powers of the homogeneous maximal ideal, and ideals generated by positive powers of the variables.
    In doing so, we effectively characterize the test ideals and $F$-jumping exponents of sufficiently general homogeneous polynomials, and of all diagonal polynomials.
    Our characterizations make these invariants computable, and show that they vary uniformly with the congruence class of the characteristic modulo a fixed integer.
    Moreover, we confirm that for a diagonal polynomial over a field of characteristic zero, the test ideals of its reduction modulo a prime agree with the reductions of its multiplier ideals for infinitely many primes. 
\end{abstract}

\maketitle

\section{Introduction}

Aiming at understanding how the Frobenius singularities of an ideal in positive characteristic relate to those of a generic element of the ideal, in \cite{hernandez+etal.frobenius_powers} the authors extended the notion of Frobenius powers, assigning to each ideal $\ideala$ of an $F$-finite regular domain $R$ of characteristic $p>0$, and each nonnegative real number~$t$, an ideal $\frob{\ideala}{t}$ of $R$.
When $t$ is an integral power of $p$, our Frobenius powers agree with the standard Frobenius powers or with the ``Frobenius roots'' of \cite{blickle+mustata+smith.discr_rat_FPTs}---that is,  $\frob{\ideala}{p^e}=\ideal{f^{p^e}:f\in \ideala}$, and  $\frob{\ideala}{1/p^e}$ is the smallest ideal $\idealb$ such that $\ideala\subseteq \frob{\idealb}{p^e}$.

Frobenius powers behave in many ways like test ideals and multiplier ideals.  
For instance, as $t$ increases, the ideals $\frob{\ideala}{t}$  form a nested, non-increasing chain, and there are only finitely many ideals of this form when we restrict $t$ to any bounded subset of nonnegative numbers.
We call the parameters at which $\frob{\ideala}{t}$ ``jumps'' the \emph{critical exponents} of $\ideala$.
These are the analogs of $F$-jumping exponents of test ideals, or jumping numbers of multiplier ideals. 
We refer the reader to \cite{hernandez+etal.frobenius_powers} for properties of Frobenius powers and critical exponents, as well as further motivation for their study, and connections with the theories of test ideals and multiplier ideals. 

The present article, which can be regarded as a companion to \cite{hernandez+etal.frobenius_powers}, focuses on the computation of Frobenius powers and critical exponents of two classes of monomial ideals: powers of the homogeneous maximal ideal, and diagonal ideals.  
Let $R=\kk[x_1,\ldots,x_n]$, where $\kk$ is a field of characteristic $p>0$ with $[ \kk : \kk^p] < \infty$. 
We summarize our results for powers of $\idealm = \ideal{x_1,\ldots,x_n}$ below.

\begin{thmintro}[\cf\ \Cref{critical exponents of m^d: T,Frobenius powers of m^d: T}]
   \label{powers of the maximal ideal: TI}
   The critical exponents of $\idealm^\D$ in the open unit interval $(0,1)$ are in correspondence with the integers $k$ satisfying $n \leq k < \D$.
   More precisely, given such a $k$, set $s = \inf \{ e \geq 1 : \lpr{kp^e}{\D} < n\}$, where $\lpr{kp^e}{\D}$ is the least \emph{positive} integer congruent to $kp^e$ modulo $\D$.
   Then the rational number
   \[ \lambda =  \frac{k}{\D} - \frac{\lpr{kp^s}{\D}}{\D p^s} \]
   is a critical exponent of $\idealm^\D$, where we interpret this formula to agree with $k/\D$ if $s= \infty$.\footnote{The rational number $\lambda$ is often called the \emph{$s$-th truncation of $k/\D$ \textup(base $p$\textup)}; see, \eg \cite[Definition~2.2, Lemma~2.5]{hernandez+others.fpt_quasi-homog.polys} or \cite[Definition~2.3, Remark~2.6]{ftf1}.}
   Moreover, every critical exponent $\lambda$ of $\idealm^\D$ in the open unit interval is as above for some $n \leq k < \D$, and if $p>\D$, then $\frob{(\idealm^\D)}{\lambda} = \idealm^{k-n+1}$ for each such $\lambda$.
\end{thmintro}

Some comments are in order:
Though \Cref{powers of the maximal ideal: TI} describes only the Frobenius powers of $\idealm^\D$ with exponents in the open unit interval, Skoda's Theorem for Frobenius powers \cite[Corollary~3.17]{hernandez+etal.frobenius_powers} implies that these ideals determine $\frob{(\idealm^\D)}{t}$ for all nonnegative $t$.
In particular, \Cref{powers of the maximal ideal: TI} tells us that the critical exponents of $\idealm^\D$ are simply the positive integers when $\D \leq n$.
Perhaps the most interesting aspect of \Cref{powers of the maximal ideal: TI} is that it illustrates that the Frobenius powers of $\idealm^\D$ vary uniformly with the class of $p$ modulo $\D$, at least when $p>\D$.

We now turn our focus to the Frobenius powers of the balanced diagonal ideal $ \ideald=\ideal{x_1^\D,\ldots,x_n^\D}$. 
Though $\ideald$ and $\idealm^\D$ share the same integral closure, it turns out that their Frobenius powers can be quite different;  contrast this with the situation for test ideals and multiplier ideals, which are insensitive to taking integral closure. To describe this difference, we recall some terminology:
If $k$ is a positive integer, then a \emph{composition of $k$ of size $n$} is a point in $\NN^n$ with positive coordinates that sum to $k$.
Roughly speaking, we show that when $p$ is large relative to $\D$ and $n$, the critical exponents of $\ideald$ in the open unit interval are determined by the compositions of size $n$ of all integers $k$ satisfying $n \leq k \leq \D$;  see \Cref{critical exponents of diagonal ideal: T} for a more precise statement.
This suggests that the ideal $\ideald$ can possess many more critical exponents in the open unit interval than $\idealm^\D$, and this is indeed often the case.

In fact, rather than restricting ourselves to the balanced case, we completely describe the critical exponents and Frobenius powers of an arbitrary diagonal ideal $\ideald=\ideal{x_1^{d_1},\ldots,x_n^{d_n}}$, where each $d_i$ is a positive integer.
As in the balanced case, we show that the critical exponents of $\ideald$ are determined by certain compositions of integers, at least when $p$ is large relative to $n$ and the exponents appearing in $\ideald$.
Furthermore, we provide an explicit description for the Frobenius powers $\frob{\ideald}{t}$.
In analogy with $\idealm^\D$, these descriptions demonstrate that when $p \gg 0$, the Frobenius powers of $\ideald$ vary uniformly with the class of $p$ modulo a fixed positive integer.
We refer the reader to \Cref{crits of general diagonal ideal: T,Frobenius powers of general diagonal ideal: T} for more precise statements.

We stress that our results allow for the effective computation of the Frobenius powers and critical exponents of diagonal ideals and powers of the homogeneous maximal ideal. 
See \Cref{m^d examples: SS,diagonal examples: SS} for examples of these computations. 

Frobenius powers and critical exponents of an ideal are closely related to test ideals and $F$-jumping exponents of generic linear combinations of generators of the ideal, and from this perspective, this paper might as well have been called \emph{Test ideals and $F$-jumping exponents of generic homogeneous and diagonal polynomials}.
Indeed, the critical exponents and Frobenius powers computed in \Cref{powers of the maximal ideal: TI} give the $F$-jumping exponents and test ideals of sufficiently general homogeneous polynomials.
For example, we have the following.

\begin{thmintro}[\cf\ \Cref{generic homogeneous poly: C}]
   \label{generic homogeneous: TI}
   Let $x^{\vv{a}_1}, \ldots, x^{\vv{a}_l}$ be the distinct monomials of degree $\D$ in the variables $x_1, \ldots, x_n$, and $\kk=\FF_p(\alpha_1,\ldots,\alpha_l)$, where $\alpha_{1},\ldots,\alpha_{l}$ are algebraically independent over $\FF_p$.
   Consider the polynomial $f = \alpha_1  x^{\vv{a}_1} + \cdots + \alpha_l x^{\vv{a}_l}\in \kk[x_{1},\ldots,x_{n}]$.
   If $p>\D$, then the test ideal $\tau(f^{t})$ is equal to $\frob{(\idealm^\D)}{t}$ for every $0 < t < 1$.
\end{thmintro}

We also have a statement relating the critical exponents of $\idealm^\D$ to the $F$-jumping exponents of $f$ as in \Cref{generic homogeneous: TI}, but for arbitrary $\kk$
(also see \Cref{generic homogeneous poly: C}).

For a diagonal polynomial $g$, that is, any $\kk^{\ast}$-linear combination of the minimal monomial generators of the ideal $\ideald$ considered above, we obtain a simpler and stronger result: the critical exponents and Frobenius powers of $\ideald$ are the $F$-jumping exponents and test ideals of the diagonal polynomial $g$; see \Cref{diagonal poly: P}.
Note that this significantly extends the results obtained in \cite{hernandez.diag_hypersurf}.

If $\ideala$ is an ideal in characteristic zero, let $\ideala_p$ and $\mathcal{J}(\ideala^t)_p$ denote the reductions of $\ideala$ and of the multiplier ideal $\mathcal{J}(\ideala^t)$ to characteristic $p \gg 0$. 
It has been conjectured that there are infinitely many primes $p$ such that $\mathcal{J}(\ideala^t)_p$ and the test ideal $\tau(\ideala_p^t)$ agree at all parameters $t$ (see, \eg \cite[Conjecture~1.2]{mustata+srinivas.ordinary_varieties}).
The fact that the test ideals of diagonal polynomials agree with the Frobenius powers of their term ideals at all parameters in the open unit interval allows us to verify this  conjecture for diagonal polynomials.

\begin{thmintro}[\cf\ \Cref{test=multiplier for diagonal: T}]
    If $g$ is a diagonal polynomial over $\QQ$, then there exist infinitely many primes $p$ such that $ \tau(g_p^{t}) = \mathcal{J}(g^{t})_p$ for every $t \geq 0$.  
\end{thmintro}

\subsubsection*{Setting and conventions}

Throughout this paper, $p$ is a positive prime integer, and $q$ will always denote a power $p^{e}$ of $p$, where $e$ is a (variable) positive integer.  

The coordinates of a point $\vv{u} \in \RR^n$ are denoted $ \vv{u} = (u_1, \ldots, u_n)$; likewise, $\vv{0} = (0,\ldots, 0)$ and $\one = (1,\ldots, 1)$.
Inequalities between points in $\RR^n$ should be interpreted as a system of $n$ coordinatewise inequalities.
Similarly, we extend standard operations on numbers to points in a coordinatewise manner.  For instance, if $\vv{u} \in \RR^n$, then $\up{\vv{u}}=(\up{u_1},\ldots,\up{u_n})$, and if $\vv{u} \in \mathbb{Z}^n$ and $d \in \mathbb{Z}$ is positive, then
\[\lpr{\vv{u}}{\D}=(\lpr{u_1}{\D},\ldots,\lpr{u_n}{\D}),\]
where $\lpr{m}{\D}$ is the \emph{least positive residue} of an integer $m$ modulo $\D$.  

The (taxicab) \emph{norm} of $\vv{u}\in \NN^n$ is the number $\norm{\vv{u}}\coloneqq u_1 + \cdots + u_n$.  If $k$ is a positive integer, then a \emph{composition of $k$ of size $n$} is a point $\vv{u} \in \NN^n$ with $\vv{u} > \vv{0}$ and $\norm{\vv{u}} = k$.
The set of all compositions of $k$ of size $n$ is denoted $\comp(k,n)$.

The \emph{multinomial coefficient} associated to $k \in \NN$ and $\vv{u} \in \NN^n$ is  defined by
\[ \binom{k}{\vv{u}} =  \frac{k!}{u_1 ! \cdots u_n!} \]
when $\norm{\vv{u}} = k$, and we adopt the convention that $\binom{k}{\vv{u}}=0$ otherwise.

Finally, $\kk$ is a field of characteristic $p$ with $[\kk: \kk^p] < \infty$, and $R = \kk[x_1, \ldots, x_n]$ is a polynomial ring over $\kk$ with homogeneous maximal ideal $\idealm = \ideal{x_1, \ldots, x_n }$.
Given a point $\vv{u} \in \NN^n$, we write $x^{\vv{u}} = x_1^{u_1} \cdots x_n^{u_n}$, and if $\vv{u}> \vv{0}$, then we call
	\[
	  \diag(\vv{u}) = \ideal{x_1^{u_1}, \ldots, x_n^{u_n}} 
	\]
the \emph{diagonal ideal} associated to $\vv{u}$.  Notice that $\frob{\diag(\vv{u})}{q} = \diag(\vectorscalar{q}{\vv{u}})$.

\section{Frobenius powers and test ideals}

Here, we recall some basic facts regarding Frobenius powers and test ideals, emphasizing the aspects most relevant to this paper.  
For a more comprehensive overview, we refer the reader to \cite{hernandez+etal.frobenius_powers} and \cite{blickle+mustata+smith.discr_rat_FPTs}.  
In what follows, let $\ideala$ be a nonzero proper ideal of the polynomial ring $R$.

The \emph{Frobenius powers} and the \emph{test ideals} of $\ideala$ are families of ideals $\frob{\ideala}{t}$ and $\tau(\ideala^t)$ of $R$ parametrized by a nonnegative real number $t$.
Both $\frob{\ideala}{t}$ and $\tau(\ideala^t)$ equal $R$ when $t$ is sufficiently small, and are nested and non-increasing as $t$ increases.  
These families are also locally constant to the right, that is, $\frob{\ideala}{t}=\frob{\ideala}{s}$ and $\tau(\ideala^t) = \tau(\ideala^s)$ for every $t\geq 0$ and $s > t$ sufficiently close to $t$.

We now recall the definition of Frobenius powers.
Starting with integral exponents, if $m\in \NN$ has base $p$ expansion $m=m_0+m_1p+\cdots+m_rp^r$, then
	\[\frob{\ideala}{m}\coloneqq\ideala^{m_0}\frob{(\ideala^{m_1})}{p}\cdots 
		\frob{(\ideala^{m_r})}{p^r}.\]
If $\ideala = \ideal{h_1, \ldots, h_l}$, then $\frob{\ideala}{m}$ is the ideal of $R$ generated by all  $h^{\vv{s}} = h_{1}^{s_1} \cdots h_{l}^{s_l}$, where $\binom{m}{\vv{s}} \not\equiv 0 \bmod p$; see \cite[Proposition~3.5]{hernandez+etal.frobenius_powers}.  

If $t=m/q$, with $m\in \NN$, then $\frob{\ideala}{t}=\frob{(\frob{\ideala}{m})}{1/q}$.
Finally, for arbitrary $t\geq 0$, like test ideals, $\frob{\ideala}{t}$ is defined by approximating $t$ on the right by rational numbers of the above type. 
Explicitly, if $(t_k)$ is a sequence of rational numbers whose denominators are powers of $p$, and $(t_{k})$ converges to $t$ monotonically from above, then $\frob{\ideala}{t}=\bigcup_k \frob{\ideala}{t_k}=\frob{\ideala}{t_k}$ for $k\gg 0$.

Clearly, $\frob{\ideala}{m}\subseteq \ideala^m$ for each $m\in \NN$, and equality holds if $m<p$ or $\ideala$ is principal.
This implies that $\frob{\ideala}{t} \subseteq \tau(\ideala^t)$ for all $t\geq 0$, and equality holds when $\ideala$ is principal.   

We refer to the parameter values where the Frobenius powers ``drop'' as \emph{critical exponents} of $\ideala$.    
Explicitly, $\lambda > 0$ is a critical exponent of $\ideala$ if $\frob{\ideala}{\lambda - \varepsilon}$ properly contains $\frob{\ideala}{\lambda}$ for every $0 < \varepsilon \leq \lambda$. 
Analogously, the parameter values where the test ideals ``drop'' are called \emph{$F$-jumping exponents} of $\ideala$. 
Both the critical exponents and the $F$-jumping exponents of $\ideala$ form discrete sets of rational numbers; see \cite[Corollary~5.8]{hernandez+etal.frobenius_powers} and \cite[Theorem~3.1]{blickle+mustata+smith.discr_rat_FPTs}.

We will use the following concrete description of critical exponents. 
If $\idealb$ is a proper ideal of $R$ with $\ideala \subseteq \sqrt{\idealb}$,  let
 \[
 \mu(\ideala, \idealb, q) \coloneqq 
	 \max \big\{ m \in \NN \mid \frob{\ideala}{m} \not\subseteq \frob{\idealb}{q} \big\},
 \]
which is a well-defined integer.  
The limit
\[
\crit(\ideala, \idealb) \coloneqq \lim_{q \to \infty} \frac{\mu(\ideala, \idealb, q)  }{q} 
\]	
exists, and is called the \emph{critical exponent of $\ideala$ with respect to $\idealb$}.  
This limit is indeed a critical exponent of $\ideala$, as demonstrated by the equalities
\[ \crit(\ideala, \idealb) = \sup \big\{ t > 0 : \frob{\ideala}{t} \not \subseteq {\idealb}  \big\}
= \min \big\{ t > 0 : \frob{\ideala}{t} \subseteq \idealb \big\}.\]
Conversely, every critical exponent $\lambda$ of $\ideala$ is of the form $\crit(\ideala, \idealb)$ for some $\idealb$.  
Indeed, one may take $\idealb =\frob{\ideala}{\lambda}$.  

An analogous description exists for $F$-jumping exponents---in this context, called \emph{$F$-thresholds}---where $\mu(\ideala,\idealb,q)$ is replaced with
 \[
 \nu(\ideala, \idealb, q) \coloneqq 
 \max \big\{ m \in \NN \mid \ideala^{m} \not\subseteq \frob{\idealb}{q} \big\}.
 \]

\begin{proposition}
   \label{min comparison: P}
   If $\ideala  \subseteq \idealb$, then $\mu(\ideala, \idealb, p) = \min \{ \nu(\ideala, \idealb, p), p-1 \}$.
\end{proposition}

\begin{proof}
Set $\mu = \mu(\ideala, \idealb, p)$ and $\nu = \nu(\ideala, \idealb, p)$.  
If $\nu \geq p-1$, then $\frob{\ideala}{p-1} = \ideala^{p-1} \not \subseteq \frob{\idealb}{p}$, while the assumption that $\ideala \subseteq \idealb$ implies that $\frob{\ideala}{p} \subseteq \frob{\idealb}{p}$, showing that $\mu = p-1$.  
On the other hand, if $\nu \leq p-1$, then $\frob{\ideala}{\nu} = \ideala^{\nu} \not \subseteq \frob{\idealb}{p}$, while  $\frob{\ideala}{\nu+1} \subseteq \ideala^{\nu+1}  \subseteq \frob{\idealb}{p}$, so $\mu = \nu$.
\end{proof}

\begin{proposition}
\label{relating mu's: P} 
If $\ideala \subseteq \sqrt{\idealb}$, then $\mu(\ideala, \idealb, q p^e) = \mu(\ideala, \idealb, q) \cdot p^e + E$ for some integer $E$ satisfying $0 \leq E \leq p^e-1$. 
\end{proposition}

\begin{proof}
	If $m = \mu(\ideala, \idealb, q)$, then $\frob{\ideala}{m} \not \subseteq \frob{\idealb}{q}$, and so $\frob{\ideala}{mp^e} = 		
                \frob{(\frob{\ideala}{m})}{p^e} \not \subseteq \frob{\idealb}{qp^e}$, due to the flatness of the Frobenius map over $R$.
	On the other hand, $\frob{\ideala}{m+1} \subseteq \frob{\idealb}{q}$, so that 
		$\frob{\ideala}{mp^e + p^e} = \frob{(\frob{\ideala}{m+1})}{p^e} \subseteq \frob{\idealb}{qp^e}$.
\end{proof}

The remainder of this section, and article, concerns the test ideals, Frobenius powers, and related numerical invariants associated to monomial ideals.  
A fact that will be frequently used throughout is that Frobenius powers of monomial ideals are themselves monomial ideals.
This is easy to see, by observing that integral Frobenius powers and Frobenius roots both have like property. 

Before stating the next result, we introduce some simplifying notation.

\begin{notation}  
	Consider $\vv{u} \in \NN^n$ with $\vv{u}>\vv{0}$.
	When dealing with any of the numerical invariants discussed above that involve $\diag(\vv{u}) = \ideal{x_1^{u_1}, \ldots, x_n^{u_n}}$, we will replace all occurrences of $\diag(\vv{u})$ in the notation with~$\vv{u}$.  
	For instance, we write $\mu(\ideala, \vv{u}, q)$ instead of $\mu(\ideala, \diag(\vv{u}), q)$.  
\end{notation}

The following may be regarded as a refinement of \Cref{relating mu's: P}.

\begin{proposition}
\label{building up: P}
If $\vv{u}$ and $\vv{v}$ are points in $\NN^n$ with positive coordinates, 
and $m \in \NN$ with $x^{\vectorscalar{q}{\vv{u}}-\vv{v}} \in \frob{\ideala}{m}$, then 
$\mu(\ideala,\vv{u}, qp^e) \geq m p^e + \min \{ p^e-1, \mu(\ideala, \vv{v}, p^e) \}$.
\end{proposition}

\begin{proof}
    Set $f = x^{\vectorscalar{q}{\vv{u}}-\vv{v}} \in \frob{\ideala}{m}$, and let $l$ be the minimum appearing in the statement.
    As $l \leq \mu(\ideala, \vv{v}, p^e)$, there exists $g \in \frob{\ideala}{l}$ with $g \notin \diag( \vectorscalar{p^e}{\vv{v}})$.
    As $l \leq p^e-1$, we have that $\frob{\ideala}{mp^e+l} = \frob{\ideala}{mp^e} \frob{\ideala}{l}$, which contains the polynomial $h = f^{p^e} g$.
    However, $h\notin \diag(\vectorscalar{qp^e}{\vv{u}})$, so $\frob{\ideala}{mp^{e}+l} \not\subseteq \diag(\vectorscalar{qp^e}{\vv{u}})$, and consequently $\mu(\ideala, \vv{u}, q p^e) \geq m p^e + l$.
\end{proof}

The following proposition characterizes the monomials in a Frobenius power of a monomial ideal in terms of critical exponents with respect to diagonal ideals.

\begin{proposition}\label{crits of mon ideals: P}
	If $\ideala$ is a monomial ideal and $\vv{u}\in \NN^n$, then $x^\vv{u}\in \frob{\ideala}{\lambda}$ if and only if $\crit(\ideala,\vv{u}+\one)>\lambda$.	
\end{proposition}

\begin{proof}
	The forward implication is immediate, since $x^\vv{u}\notin \diag(\vv{u}+\one)$.
	Conversely, if $\crit(\ideala, \vv{u}+\one) > \lambda$, then $\frob{\ideala}{\lambda}\not\subseteq\diag(\vv{u}+\one)$, so there exists a monomial $x^{\vv{v}} \in \frob{\ideala}{\lambda}$ with $\vv{v} < \vv{u} + \one$.
	Equivalently, $\vv{v} \leq \vv{u}$, hence $x^{\vv{u}}\in \ideal{x^{\vv{v}}}\subseteq\frob{\ideala}{\lambda}$. 
\end{proof}

The following result clarifies the task of computing all of the critical exponents of certain monomial ideals.

\begin{proposition}
\label{crits of m-primary ideal: P}
	Every critical exponent of an $\idealm$-primary monomial ideal $\ideala$ is of the form 
		$\crit(\ideala, \vv{u})$, where $\vv{u}$ is a point in $\NN^n$ with $\vv{u} > \vv{0}$.
\end{proposition}

\begin{proof}  
	If $\lambda$ is a critical exponent of $\ideala$, then $\lambda = \crit(\ideala,\idealb)$, with $\idealb = \frob{\ideala}{\lambda}$ a proper monomial ideal.
	Moreover,  $\idealb = \frob{\ideala}{\lambda} \supseteq \frob{\ideala}{q}$ for every $q > \lambda$, and so $\idealb$ must also be $\idealm$-primary.  
	Being a monomial ideal, $\idealb$ is a finite intersection of ideals generated by powers of the variables \cite[Lemma~5.18]{miller+sturmfels.combinatorial_CA}. 
	As $\idealb$ is $\idealm$-primary, the intersectands must also be $\idealm$-primary, and therefore diagonal ideals, so $\idealb = \diag(\vv{u}_1) \cap \cdots \cap \diag(\vv{u}_l)$, for some $\vv{u}_1,\ldots,\vv{u}_l>\vv{0}$. 
	As standard Frobenius powers commute with intersections in our polynomial ring, it is easy to see that $\mu(\ideala, \idealb, q) = \max \mu(\ideala, \vv{u}_i, q)$, whence $\lambda = \max \crit(\ideala, \vv{u}_i)$.  
\end{proof}

\section{A power of the homogeneous maximal ideal}

This section is dedicated to computing the critical exponents and Frobenius powers of the ideal $\idealm^\D$, where $\D$ is some fixed positive integer.  
We call upon the following observation in our calculation. 

\begin{lemma}
   \label{pigeonhole: L}
   If $k$ is an integer with $k \geq n$, then
   \[ \idealm^{k-n+1} = \bigcap \, \diag(\vv{u}), \]
   where the intersection is taken over all points $\vv{u} \in \comp(k,n)$.
\end{lemma}

\begin{proof}  
	Let $\vv{v} \in \NN^n$.  
	If $x^{\vv{v}} \notin \diag(\vv{u})$ for some $\vv{u} \in \comp(k,n)$, then $\vv{v} \leq \vv{u} - \one$.  
	Then $\norm{\vv{v}} \leq \norm{\vv{u}-\one} = k-n$, so that $x^{\vv{v}} \notin \idealm^{k-n+1}$.
	Conversely, if $x^{\vv{v}} \notin \idealm^{k-n+1}$, then $\norm{\vv{v}} \leq k-n$, so that 
		$\norm{\vv{v} + \one} = \norm{\vv{v}} + n \leq k$.   
	This condition guarantees the existence of a point $\vv{u} \in \comp(k, n)$ for which 
		$\vv{u} \geq \vv{v}+\one$, so that $x^{\vv{v}} \notin \diag(\vv{u})$.
\end{proof}

To simplify our notation in this section, given a point $\vv{u} \in \NN^n$ with $\vv{u} > \vv{0}$, we write $\nu(\vv{u}, q)$ and $\mu(\vv{u}, q)$ instead of $\nu(\idealm^\D, \vv{u}, q)$ and $\mu(\idealm^\D, \vv{u}, q)$, respectively.

\subsection{Critical exponents}
\label{critical exponents of m^d: subsection}

We begin by characterizing the critical exponents of $\idealm^\D$ that lie in the unit interval.  

\begin{proposition}\label{small crits of m^D: P}
	Every critical exponent of $\idealm^\D$ in the unit interval is of the form 
	$\crit(\idealm^\D, \vv{u})$, where $\vv{u} \in \NN^n$ satisfies $\vv{u} > \vv{0}$ and 
	$\norm{\vv{u}} \leq \D+n-1$.
\end{proposition}

\begin{proof}
   By \Cref{crits of m-primary ideal: P}, every critical exponent of $\idealm^\D$ is of the form $\crit(\idealm^\D, \vv{u})$, for some $\vv{u} \in \NN^n$ with $\vv{u} > \vv{0}$.
   But $\crit(\idealm^\D, \vv{u})\leq 1$  precisely when $\idealm^\D \subseteq \diag(\vv{u})$, a condition equivalent to $\norm{\vv{u}} \leq \D+n-1$, by \Cref{pigeonhole: L}.
\end{proof}

Our objective is to compute the critical exponents appearing in \Cref{small crits of m^D: P}.  We begin by first dispensing with a degenerate case, in which the  norm of the point $\vv{u}$ can be regarded as ``too large.''

\begin{proposition}
\label{degenerate m^D: P}
	If $\vv{u} \in \NN^n$ satisfies $\vv{u} > \vv{0}$ and $\D \leq \norm{\vv{u}} \leq \D+n-1$, 
		then $\mu(\vv{u}, q) = q-1$ for every $q$.  
	In particular, $\crit(\idealm^\D, \vv{u}) = 1$.
\end{proposition}

\begin{proof}  
	The upper bound on the norm of $\vv{u}$ implies that $\idealm^\D\subseteq \diag(\vv{u})$, 
		and so $\mu(\vv{u}, q) \leq q-1$.  
	To establish equality, it suffices to notice that since $x^{\vv{u}}\in \idealm^\D$, we have that $x^{\scalarvector{(q-1)}{\vv{u}}}\in \frob{(\idealm^\D)}{q-1}\setminus \diag(\vectorscalar{q}{\vv{u}})$. 
\end{proof}

\begin{corollary}
	If $\D \leq n$, then the only critical exponent of $\idealm^\D$ in $[0,1]$ is $1$.
\end{corollary}

\begin{proof}  
	If $\vv{u} \in \NN^n$ has positive coordinates, then $\norm{\vv{u}} \geq n$, 
		which is at least $\D$ by assumption.  
	Our claim then follows from \Cref{small crits of m^D: P,degenerate m^D: P}.
\end{proof}

With this in hand, it remains to compute the critical exponent of $\idealm^\D$ with respect to a point in $\NN^n$ with positive coordinates and whose norm is at most $\D$, which we accomplish in \Cref{critical exponents of m^d: T} below.
In order to simplify the upcoming discussion, we fix the following notation.

\begin{setup}
    \label{m^d: N}
    Fix an integer $k$ with $n \leq k \leq \D$ and a point $\vv{u} \in \comp(k, n)$.
\end{setup}

The following is the main result of this section.

\begin{theorem}\label{critical exponents of m^d: T}
    In the context of \Cref{m^d: N},
    if $s = \inf \{ e \geq 1 : \lpr{kp^e}{\D} < n\}$, then
    \[ \crit(\idealm^\D, \vv{u}) =  \frac{k}{\D} - \frac{\lpr{kp^s}{\D}}{\D p^s}, \]
    where we interpret this formula to agree with $k/\D$ if $s= \infty$.
    In particular, $\crit(\idealm^\D, \vv{u})$ depends only on $k$, but not on the point $\vv{u} \in \comp(k,n)$.
\end{theorem}

\begin{remark}\label{crits in open interval for m^d: R}
    \Cref{degenerate m^D: P,critical exponents of m^d: T} tell us that the only critical exponents of $\idealm^\D$ less than $1$ are those corresponding to compositions $\vv{u} \in \comp(k,n)$ with $n \leq k < \D$.
    Moreover, if $p>\D$, then for every such $k$ there is a unique critical exponent in the interval $\big(\frac{k-1}\D,\frac k\D\big]$.
\end{remark}

\begin{remark}
   The extended integer $s$ in \Cref{critical exponents of m^d: T} depends only on the congruence class of $p$ modulo $\D$.
   Moreover, as $\lpr{kp^e}{\D}$ varies periodically with $e$, the value of~$s$, and hence the value of $\crit(\idealm^\D, \vv{u})$, can always be effectively computed.  
\end{remark}

\begin{remark}
   Under the assumptions of \Cref{critical exponents of m^d: T}, if $p \equiv 1 \bmod \D$, then $\lpr{kp^e}{\D} = k \geq n$ for all $e \geq 1$; thus, $s= \infty$ and $\crit(\idealm^\D, \vv{u}) = k/\D$ for every $\vv{u} \in \comp(k, n)$.
   If $p \equiv -1 \bmod \D$ and $k<\D$, then $\lpr{kp^e}{\D} = k$ or $\D - k$, depending on the parity of $e$.
   Consequently, if $k > \D-n$, then $\lpr{kp}{\D} = \D - k < n$, so that $s = 1$ and 
\[
\crit(\idealm^\D, \vv{u}) = \frac{k}{\D} - \frac{\D-k}{\D p}.
\]
   On the other hand, if $k \leq \D - n$, then since $\D - k \geq n$ and $k \geq n$, we have that $s = \infty$ and  $\crit(\idealm^\D, \vv{u}) = k/\D$.
\end{remark}

Below, we establish a series of results that we use in our proof of \Cref{critical exponents of m^d: T}.
We continue to work in the context established in \Cref{m^d: N}.  

\begin{lemma}\label{regular powers: L}
    Set $\delta = 0$ when $\lpr{kq}{\D} \geq n$ and $\delta = -1$ otherwise.
    Then
      \[ \nu(\vv{u}, q) = \frac{kq- \lpr{kq}{\D}}{\D} + \delta < q.\]  
\end{lemma}

\begin{proof}
    Set $\lambda = (kq-n)/\D$.
    If $N > \lambda$, then $\idealm^{\D N} \subseteq \idealm^{kq - n + 1} \subseteq \diag(\vectorscalar{q}{\vv{u}})$, where the last containment follows from \Cref{pigeonhole: L}.
    However, if $N \leq  \lambda$, then $\idealm^{k q- n} \subseteq \idealm^{\D N}$, but the monomial $x^{\vectorscalar{q}{\vv{u}}-\one} \in \idealm^{k q- n}$ is not in $\diag(\vectorscalar{q}{\vv{u}})$, so $\idealm^{\D N} \not\subseteq \diag(\vectorscalar{q}{\vv{u}})$.

    The preceding argument shows that $\nu(\vv{u}, q) = \down{\lambda}$, which is less than $q$, as $k \leq d$, and the equality $\nu(\vv{u}, q) = \down{\lambda}$ can also be explicitly described as
      \[ \nu(\vv{u}, q) = \down{\frac{kq-n}{\D}} =   \frac{kq-\lpr{kq}{\D}}{\D} +  \down{\frac{\lpr{kq}{\D}-n}{\D} }, \]
      from which our claim follows.
\end{proof}

\begin{corollary}
   \label{recursive nu and mu: C}
   Set $l=\lpr{kp}{\D}$.
   If $l\geq n$, then
   \begin{equation*}
          \nu(\vv{u}, qp) = \nu(\vv{u}, p) \cdot q + \nu(\vv{v}, q)
       \end{equation*}
       for every $\vv{v} \in \comp(l, n)$.
   Moreover, there exists $\vv{v} \in \comp(l,n)$ such that \begin{equation*}
		\mu(\vv{u}, qp) \geq \nu(\vv{u}, p) \cdot q + \mu(\vv{v}, q).
	\end{equation*} 
\end{corollary}

\begin{proof}
    As $n \leq l \leq \D$, we can use \Cref{regular powers: L} to compute not only $\nu(\vv{u}, p)$ and $\nu(\vv{u}, qp)$, but also $\nu(\vv{v}, q)$ for every $\vv{v} \in \comp(l, n)$.
    Keeping in mind that $\lpr{lq}{\D} = \lpr{kqp}{\D}$, the first equation follows from comparing these formulas.

    By \Cref{min comparison: P,regular powers: L}, $\mu(\vv{u}, p) = \nu(\vv{u}, p)=(kp-l)/\D<p$, so there is a monomial in $\frob{(\idealm^\D)}{\mu(\vv{u}, p)} = \idealm^{\D \mu(\vv{u}, p)} = \idealm^{kp-l}$ not in $\diag(\vectorscalar{p}{\vv{u}})$.
    This monomial is of the form $x^{\vectorscalar{p}{\vv{u}}-\vv{v}}$ for some $\vv{v}\in \comp(l, n)$.
    As $\mu(\vv{v}, q) \leq \nu(\vv{v}, q)<q$, \Cref{building up: P} shows that $\mu(\vv{u}, qp) \geq \mu(\vv{u}, p) \cdot q + \mu(\vv{v}, q) = \nu(\vv{u}, p) \cdot q + \mu(\vv{v},q)$. 
\end{proof}

\begin{lemma}
   \label{nu=mu: L}
   If $\lpr{kp^e}{\D} \geq n$ whenever $p \leq p^e < q$, then $\mu(\vv{u},q) = \nu(\vv{u}, q)$.
\end{lemma}

\begin{proof}
    Given  $\vv{v} \in \comp(l, n)$, where  $l = \lpr{kp}{d}$, \Cref{min comparison: P,regular powers: L} imply that $\mu(\vv{v}, p) = \nu(\vv{v}, p)$.
    By way of induction, suppose that the desired statement holds for $q$, and that $\lpr{kp^e}{\D} \geq n$ whenever  $p \leq p^e < qp$.
    Setting $e=1$ shows that $l\geq n$, and $\lpr{lp^{e-1}}{\D} = \lpr{kp^e}{\D} \geq n$ whenever $1 \leq p^{e-1} < q$.
    Then for $\vv{v}$ satisfying both conclusions of \Cref{recursive nu and mu: C}, $\mu(\vv{v}, q) = \nu(\vv{v}, q)$ by our induction hypothesis.
    That corollary then implies that $\mu(\vv{u},qp) \geq \nu(\vv{u}, qp)$, hence $\mu(\vv{u},qp) = \nu(\vv{u}, qp)$.
\end{proof}

We are now ready to prove the main result of this section.

\begin{proof}[Proof of \Cref{critical exponents of m^d: T}]
    First suppose that $s < \infty$, and set $q = p^s$.
    Then $\lpr{kp^e}{\D} \geq n$ for all $p \leq p^e < q$, but $\lpr{kq}{\D} < n$, which allows us to invoke \Cref{nu=mu: L,regular powers: L} to deduce that
    \[ \mu(\vv{u}, q) = \nu(\vv{u}, q) = \frac{kq-\lpr{kq}{\D}}{\D} - 1.\]

    This expression for $\mu(\vv{u}, q)$ implies that there exists a monomial of $\frob{(\idealm^\D)}{\mu(\vv{u}, q)}$, and hence a monomial of $(\idealm^\D)^{\mu(\vv{u}, q)} = \idealm^{kq-\lpr{kq}{\D}-\D}$, not lying in $\diag(\vectorscalar{q}{\vv{u}})$.
    Such a monomial is necessarily of the form $x^{\vectorscalar{q}{\vv{u}}-\vv{v}}$ for some $\vv{v} \in \NN^n$ with $\vv{v} > \vv{0}$ and $\norm{\vectorscalar{q}{\vv{u}}-\vv{v}} = kq-\lpr{kq}{\D}-\D$, and therefore,
    \[ \norm{\vv{v}} = ( \norm{\vv{u}}-k) \cdot q + \lpr{kq}{\D} + \D  = \lpr{kq}{\D} + \D. \]
    It follows that $\D < \norm{\vv{v}} \leq \D+n-1$, and so \Cref{degenerate m^D: P} tells us that $\mu(\vv{v}, p^e) = p^e-1$ for all $e \geq 1$.
    Given this, \Cref{building up: P} with $m=\mu(\vv{u}, q)$ then states that
    \[ \mu(\vv{u}, qp^e) \geq  \mu(\vv{u}, q) \cdot p^e + p^e-1, \]
    and \Cref{relating mu's: P} implies that equality must hold, so
    \[
        \mu(\vv{u}, qp^e) = ( \mu(\vv{u}, q)  + 1) \cdot p^e - 1 = \left( \frac{kq-\lpr{kq}{\D}}{\D} \right) \cdot p^e - 1
    \]
    for every $e\geq 1$.
    Dividing by $qp^e$ and letting $e\to \infty$ gives the desired formula for $\crit(\idealm^\D,\vv{u})$.

    Next, instead suppose that $s = \infty$, or equivalently, that $\lpr{kp^e}{\D} \geq n$ for every $e \geq 1$.
    In this case, \Cref{nu=mu: L,regular powers: L} show that
    \[
        \mu(\vv{u}, p^e) = \nu(\vv{u}, p^e) = \frac{kp^e-\lpr{kp^e}{\D}}{\D}
    \]
    for every $e \geq 1$.
    Dividing by $p^e$ and letting $e\to\infty$ shows that $\crit(\idealm^\D, \vv{u})=k/\D$.
\end{proof}

\subsection{Frobenius powers}

We now turn our attention to computing the Frobenius powers $\frob{(\idealm^\D)}{\lambda}$. 
By Skoda's Theorem for Frobenius powers \cite[Corollary~3.17]{hernandez+etal.frobenius_powers}, we can assume that $\lambda$ is a critical exponent of $\idealm^\D$ less than $1$, or equivalently, $\lambda=\crit(\idealm^{\D},\vv{u})$, for some $\vv{u}>\vv{0}$ with $\norm{\vv{u}}<\D$ (see \Cref{crits in open interval for m^d: R}).
In \Cref{Frobenius powers of m^d: T} below, we compute $\frob{(\idealm^\D)}{\lambda}$ at every such $\lambda$.
However, in contrast to results in \Cref{critical exponents of m^d: subsection}, we require the additional assumption that $p \gg 0$.

\begin{theorem}
\label{Frobenius powers of m^d: T}
	Suppose $p>\D$.  
	Fix an integer $k$ with $n \leq k < \D$, and let $\lambda$ denote the common value of $\crit(\idealm^\D, \vv{u})$ for every $\vv{u} \in \comp(k,n)$.  
	Then $\frob{(\idealm^\D)}{\lambda} = \idealm^{k-n+1}$.  
\end{theorem}

\begin{proof} 
	By the definition of the critical exponent $\lambda$ and \Cref{pigeonhole: L},  
		\[ \frob{(\idealm^\D)}{\lambda} \subseteq \bigcap \diag(\vv{u}) = \idealm^{k-n+1}, \] 
	where the intersection is taken over all $\vv{u}  \in \comp(k,n)$.   
	Conversely, if $x^{\vv{u}}$ is a monomial with $\norm{\vv{u}} = k-n+1$, then $\vv{u} + \one$ has norm $k+1 \leq \D$,  and \Cref{critical exponents of m^d: T} then tells us that 
		\[ \crit(\idealm^\D, \vv{u} + \one) = \frac{k+1}{\D} - \frac{\lpr{(k+1)p^s}{\D}}{\D p^s} \] 
	for some positive extended integer $s$, where we interpret the second term to be zero if $s=\infty$.  
	However, no matter the value of $s$, the fact that $\lpr{l}{\D} \leq \D$ for every integer $l$ implies that in all cases,  
		\[ \crit(\idealm^\D, \vv{u}+\one) \geq \frac{k+1}{\D} - \frac{1}{p}  > \frac{k}{\D} \geq \lambda, \] 
	where the second-to-last inequality follows from our assumption that $p > \D$, and the last from \Cref{critical exponents of m^d: T}.
	\Cref{crits of mon ideals: P} then shows that $x^\vv{u}\in\frob{(\idealm^\D)}{\lambda}$, and as $x^{\vv{u}}$ was arbitrary, we conclude that $\idealm^{k-n+1}\subseteq\frob{(\idealm^\D)}{\lambda}$.
\end{proof} 

\subsection{Examples}
\label{m^d examples: SS}

\Cref{Frobenius powers of m^d: T} tells us that 
\[ \big\{ \frob{(\idealm^\D)}{t} : 0 \leq t < 1 \big\} = \big\{  \kk[x_1, \ldots, x_n], \idealm, \idealm^2, \ldots, \idealm^{\D-n}\big\}. \]
 Of course, the $t$-values corresponding to any given possibility depend strongly on $p$ modulo $\D$, as we see below.

\begin{example} \label{power of m simple 1: E}
	Suppose $\idealm = \ideal{x, y, z}$. If $p > 7$ and $p \equiv 6 \bmod 7$, then
		\begin{equation*}
			\frob{(\idealm^7)}{t}=
			\begin{cases} 
				\kk[x,y,z] &\text{if } t \in \big[0, \frac{3}{7} \big) \\ 		
				\idealm &\text{if } t \in \big[\frac{3}{7}, \frac{4}{7}\big) \\ 	
				\idealm^2 &\text{if } t \in \big[\frac{4}{7}, \frac{5}{7} - \frac{2}{7p} \big) \\ 		
				\idealm^3 &\text{if } t \in \big[ \frac{5}{7} - \frac{2}{7p} , \frac{6}{7} - \frac{1}{7p}\big) \\ 
				\idealm^4 &\text{if } t \in \big[ \frac{6}{7} - \frac{1}{7p}, 1 \big) \\ 
		\end{cases}
		\end{equation*}
	In this case, we find that the value of the extended integer $s$ appearing in  \Cref{critical exponents of m^d: T} is either infinite or equals $1$.  
	In particular, this suggests that $s$ usually depends on the value of $k$.

\end{example}

\begin{example} \label{power of m simple 2: E}
  As in \Cref{power of m simple 2: E}, consider the ideal $\idealm = \ideal{x, y, z}$.
  If $p > 7$ and $p \equiv 5 \bmod 7$, then
  \begin{equation*}
			\frob{(\idealm^7)}{t}=
			\begin{cases} 
				\kk[x,y,z] &\text{if } t \in \big[0, \frac{3}{7} - \frac{1}{7p} \big) \\ 		
				\idealm &\text{if } t \in \big[\frac{3}{7} - \frac{1}{7p}, \frac{4}{7} - \frac{2}{7p^2} \big) \\ 	
				\idealm^2 &\text{if } t \in \big[\frac{4}{7} - \frac{2}{7p^2} , \frac{5}{7} - \frac{2}{7p^3} \big) \\ 		
				\idealm^3 &\text{if } t \in \big[ \frac{5}{7} - \frac{2}{7p^3} , \frac{6}{7} - \frac{2}{7p}\big) \\ 
				\idealm^4 &\text{if } t \in \big[ \frac{6}{7} - \frac{2}{7p}, 1 \big) \\ 
		\end{cases}
              \end{equation*}
  Here, $s$ is finite for all $k$, and takes on the values $1,2$, and $3$ as $k$ varies.
\end{example}

\section{A balanced diagonal ideal}
\label{diagonalsBalanced: S}

In this section, we determine the critical exponents of the balanced diagonal ideal
  \[
  \ideald = \diag(\scalarvector{\D}{\one}) = \ideal{x_1^{\D}, \ldots, x_n^{\D}}, \]
where $\D$ is a positive integer.
We compute the critical exponents and Frobenius powers of a general diagonal ideal in the next section.

\begin{convention}
   In this section, the expression $p \gg 0$ means both that $p > n\D-n$, and that $p \nmid \D$.  
\end{convention}

\Cref{crits of m-primary ideal: P} tells us that all critical exponents of $\ideald$ in the unit interval have the form $\crit(\ideald, \vv{u})$, where $\vv{u} \in \NN^n$ is a point with $\vv{u} > \vv{0}$, and such that $\ideald$ lies in $\diag(\vv{u})$, or equivalently, such that $\vv{u} \leq \scalarvector{\D}{\one}$.
Furthermore, it is not difficult to show that if some coordinate of $\vv{u}$ equals $\D$, then the corresponding critical exponent must be $1$.  We summarize these observations.

\begin{proposition} \label{critical exponents of diagonal can be taken with respect to diagonal ideals: P}
  Every critical exponent of $\ideald =  \ideal{ x_1^{\D}, \ldots, x_n^{\D}}$ in the unit interval has the form $\crit(\ideald, \vv{u})$, where the point $\vv{u}\in \NN^n$ satisfies $\vv{0} < \vv{u} \leq \scalarvector{\D}{\one}$.
  In addition, if $\vv{u} \not < \scalarvector{\D}{\one}$, then $\crit(\ideald, \vv{u}) = 1$.
  \qed
\end{proposition}

Toward the desired computation, we begin with a useful characterization of $\mu(\ideald, \vv{u}, q)$.
As $\ideald$ is fixed throughout this section, we omit it from the notation, and write $\mu(\vv{u}, q)$ rather than $\mu(\ideald,\vv{u}, q)$.  

\begin{lemma}
\label{mu for diagonal: L}
  If $\vv{u} \in \NN^n$ satisfies $\vv{u} > \vv{0}$, then $\mu(\vv{u}, q)$ equals
\[ \max \left \{ \norm{\vv{s}} : \vv{s} \in \NN^n \text{ with } \binom{\norm{\vv{s}}}{\vv{s}} \not \equiv 0 \bmod p \text{ and } \vv{s} \leq \frac{\vectorscalar{q}{\vv{u}}- \lpr{\vectorscalar{q}{\vv{u}}}{\D}}{\D} \right \}.\]
\end{lemma}

\begin{proof}
   Recall that if $m\in \NN$, then $\frob{\ideald}{m}$ is generated by monomials $x^{\scalarvector{\D}{\vv{s}}}$ such that $\norm{\vv{s}}=m$ and $\binom{\norm{\vv{s}}}{\vv{s}}\not\equiv 0 \bmod p$.
   The monomial $x^{\scalarvector{\D}{\vv{s}}}$ does not lie in $\frob{\diag(\vv{u})}{q} = \diag(\vectorscalar{q}{\vv{u}})$ if and only if $\scalarvector{\D}{ \vv{s}} < \vectorscalar{q}{\vv{u}}$, which is equivalent to 
\[ \vv{s} < \frac{\vectorscalar{q}{\vv{u}}}{\D} \iff \vv{s} \leq \up{\frac{\vectorscalar{q}{\vv{u}}}{\D}} - \one = \frac{\vectorscalar{q}{\vv{u}} - \lpr{\vectorscalar{q}{\vv{u}}}{\D}}{\D}  + \up{   \frac{\lpr{\vectorscalar{q}{\vv{u}}}{\D}}{\D}} - \one.
\]
   Finally, by the definition of least positive residue, $\up{ \lpr{\vectorscalar{q}{\vv{u}}}{\D}/\D} = \one$. 
\end{proof}

With this in hand, we now dispense with a degenerate case.

\begin{proposition} 
\label{degenerate diagonal crit: P}
  If $p\gg 0$ and $\vv{u} \in \NN^n$ satisfies $\vv{0} < \vv{u} \leq \scalarvector{\D}{\one}$ and $\norm{\vv{u}} > \D$, then $\mu(\vv{u},p^{e})=p^{e}-1$, for every $e\geq 1$, and consequently $\crit(\ideald, \vv{u}) = 1$. 
\end{proposition}

\begin{proof}
  By assumption, $\ideald \subseteq \diag(\vv{u})$, and so $\mu(\vv{u}, p^e) \leq p^e-1$ for every $e \geq 1$.
  \Cref{critical exponents of diagonal can be taken with respect to diagonal ideals: P} allows us to assume that all coordinates of $\vv{u}$ are less than $\D$, and as
  $p \gg 0$, the same must be true for $\lpr{\vectorscalar{p}{\vv{u}}}{\D}$.
  Given this, our choice of $p \gg 0$
  also guarantees that $\norm{\lpr{\vectorscalar{p}{\vv{u}}}{\D}}$ is at most $p$, and our assumption that $\norm{\vv{u}} > \D$ then implies that $\norm{(\vectorscalar{p}{\vv{u}}-\lpr{\vectorscalar{p}{\vv{u}}}{\D})/ \D} \geq p$.

  This observation guarantees the existence of a point $\vv{t} \in \NN^n$ with norm $p-1$, and such that $\vv{t} \leq (\vectorscalar{p}{\vv{u}}-\lpr{\vectorscalar{p}{\vv{u}}}{\D}) /\D$, where this equality must be strict in at least one coordinate.
  If $i$ is the index of such a coordinate, and $\canvec_i$ is the corresponding standard basis vector, then the reader can verify that
  \[ \vv{s} = \vectorscalar{p^{e-1}}{\vv{t}} +  \scalarvector{(p^{e-1}-1)}{\canvec_i} \]
  satisfies $\norm{\vv{s}} = p^{e}-1, \binom{p^{e}-1}{\vv{s}} \not \equiv 0 \bmod p$ (see \cite{dickson.multinomial}), and $\vv{s} \leq (\vectorscalar{p^e}{\vv{u}}-\lpr{\vectorscalar{p^e}{\vv{u}}}{\D}) /\D$.
  Given this, \Cref{mu for diagonal: L} implies that $\mu(\vv{u}, p^{e}) \geq p^e-1$, and equality holds. 
\end{proof}

Given \Cref{critical exponents of diagonal can be taken with respect to diagonal ideals: P,degenerate diagonal crit: P}, to understand the critical exponents of $\ideald$, it remains to determine $\crit(\ideald, \vv{u})$ in the following context.

\begin{setup}
   \label{balanced diag: setup}
  Fix a point $\vv{u} \in \NN^n$ with $\vv{u} > \vv{0}$ and $\norm{\vv{u}} \leq \D$.
  Equivalently, $\vv{u} \in \comp(k,n)$, where $n \leq k \leq \D$. 
\end{setup}

\begin{remark}
  \label{when norm and lpr commute: R}
  In our upcoming arguments, we are often concerned with the difference between  $\norm{\lpr{\vectorscalar{q}{\vv{u}}}{\D}}$, a sum of least positive residues, and $\lpr{kq}{\D}$, the least positive residue of the norm  $\norm{\vectorscalar{q}{\vv{u}}} = kq$.
  Clearly, these integers are congruent modulo $\D$, and the latter is the least positive residue of the former.
  Furthermore, it is easy to see that 
\[ \norm{\lpr{\vectorscalar{q}{\vv{u}}}{\D}} = \lpr{kq}{\D} \iff \norm{\lpr{\vectorscalar{q}{\vv{u}}}{\D}} \leq \D.\]
\end{remark}

\begin{lemma}
   \label{canonical comp: L}
   If $\norm{\lpr{\vectorscalar{p^e}{\vv{u}}}{\D}} \leq \D$ for every $p \leq p^e < q$, then if
   \[ \vv{s} = \frac{ \vectorscalar{q}{\vv{u}}-\lpr{\vectorscalar{q}{\vv{u}}}{\D}}{\D}\]
   and $p>\D$, then $\binom{\norm{\vv{s}}}{\vv{s}} \not\equiv 0 \bmod p$.
\end{lemma}

\begin{proof}
   If $q=p^r$, then
   \[\vv{s} = \frac{\vectorscalar{q}{\vv{u}}-\lpr{\vectorscalar{q}{\vv{u}}}{\D}}{\D} = \sum_{e=1}^r
    \left( \, \frac{   \lpr{\vectorscalar{p^{e-1}}{\vv{u}}}{\D} \cdot p -
           \lpr{\vectorscalar{p^e}{\vv{u}}}{\D}}{\D} \, \right) \cdot p^{r-e}.\]
 The fact that $\norm{\lpr{\vv{u}}{\D}} \leq \norm{\vv{u}} \leq \D$ and our assumption that $\norm{\lpr{ \vectorscalar{p^{e-1}}{\vv{u}}}{\D}} \leq \D$ for all $1 < e \leq r$ imply that the coordinates of $\vv{s}$ add in base $p$ without carrying, which is equivalent to the condition that $\binom{\norm{\vv{s}}}{\vv{s}} \not\equiv 0 \bmod p$ \cite{dickson.multinomial}.
\end{proof}

We now establish the main result of this section.
Besides characterizing the critical exponents of a balanced diagonal ideal, it also shows that, like in the case of a power of the maximal ideal, these numbers can be effectively computed.

\begin{theorem}  \label{critical exponents of diagonal ideal: T}
   Adopt the context of \Cref{balanced diag: setup}.  If $p \gg 0$ and  \[ s = \inf \{ e \geq 1 :  \norm{\lpr{\vectorscalar{p^e}{\vv{u}}}{\D}} > \D \}\]  then 
\[ \crit(\ideald, \vv{u}) = \frac{k}{\D} -  \frac{\norm{\lpr{\vectorscalar{p^s}{\vv{u}}}{\D}} - \D}{\D p^s}, \]
where we interpret this formula to agree with $k/\D$ if $s = \infty$.
\end{theorem}

\begin{proof}
   Set $q=p^s$.
   If $q  < \infty$, then \Cref{canonical comp: L} allows us to conclude that
   \[ \vv{s} = \frac{\vectorscalar{q}{\vv{u}}-\lpr{\vectorscalar{q}{\vv{u}}}{\D}}{\D} \]
   achieves the maximum appearing in the statement of \Cref{mu for diagonal: L}, so that
   \[ \mu(\vv{u}, q) = \norm{\vv{s}} = \frac{kq- \norm{\lpr{\vectorscalar{q}{\vv{u}}}{\D}}}{\D}.\]
   In addition, if we set $\vv{v} = \lpr{\vectorscalar{q}{\vv{u}}}{\D}$, then
   \begin{equation}
\label{left over: e}
x^{\vectorscalar{q}{\vv{u}}-\vv{v}} = x^{\scalarvector{\D}{\vv{s}}} \in \frob{\ideald}{\mu(\vv{u}, q)}.
\end{equation}
   Since $\norm{\vv{v}} > \D$ by the definition of $q$, \Cref{degenerate diagonal crit: P} shows that $\mu(\vv{v}, p^e) = p^e-1$ for every $e \geq 1$.
   This observation, along with \eqref{left over: e}, enables us to apply \Cref{building up: P} to see that $\mu(\vv{u}, qp^e)\geq \mu(\vv{u}, q) \cdot p^e + p^e-1$, and \Cref{relating mu's: P} then implies that equality must hold.
   In other words, if $e \geq 1$, then
\[ \mu(\vv{u}, qp^e) =   ( \mu(\vv{u}, q) + 1 )\cdot  p^e  -1 = \left( \frac{ kq - \norm{\lpr{\vectorscalar{q}{\vv{u}}}{\D}}}{\D} + 1 \right) \cdot p^e - 1. \] 

   Finally, if $q = \infty$, then \Cref{when norm and lpr commute: R} tells us that $\norm{\lpr{\vectorscalar{p^e}{\vv{u}}}{\D}} = \lpr{kp^e}{\D}$ for every $e \geq 1$, and a straightforward application of \Cref{mu for diagonal: L,canonical comp: L} then implies that
   \[ \mu(\vv{u}, p^e) = \frac{kp^e-\lpr{kp^e}{\D}}{\D} \]
   for all $e \geq 1$.
   In either case, the formula for $\crit(\ideald,\vv{u})$ follows.
\end{proof}

\begin{remark} \label{comparing critical exponents of balanced diagonal: remark}
   Under the hypotheses of \Cref{critical exponents of diagonal ideal: T}, if $s<\infty$, then we have that $\norm{\lpr{\vectorscalar{p^s}{\vv{u}}}{\D}} -\D \leq n(\D-1) < p \leq p^s$.
   The theorem then shows that
   \[\frac{k-1}\D < \crit(\ideald, \vv{u}) \leq \frac k \D.\]
   This separates the critical exponents $\crit(\ideald, \vv{u})$ by norm: if $\norm{\vv{u}} < \norm{\vv{v}}$, then  $\crit(\ideald, \vv{u}) < \crit(\ideald, \vv{v})$.
   If $\norm{\vv{u}} = \norm{\vv{v}}$, on the other hand,  the relationship between $\crit(\ideald, \vv{u})$ and $\crit(\ideald, \vv{v})$ depends only on the class of $p$ modulo $\D$.
\end{remark}

\begin{remark}  \label{balanced diagonal special congruences: remark}
\Cref{critical exponents of diagonal ideal: T} implies that if $n \geq 2$, then 
\[
\crit(\ideald,\vv{u}) = 
	\begin{cases}
		\frac{k}{\D} &\text{if } p \gg 0 \text{ and }  p \equiv 1 \bmod \D \\
		\frac{k}{\D} - \frac{n\D-k-\D}{\D p}& \text{if } p \gg 0 \text{ and } p \equiv -1 \bmod \D
	\end{cases}
\]
In particular, if $p\equiv \pm 1\bmod \D$, then $\crit(\ideald, \vv{u})$ depends only on $k = \norm{\vv{u}}$.
Indeed, if $p \equiv 1 \bmod \D$, then for all $e \geq 1$, $\lpr{\vectorscalar{p^e}{\vv{u}}}{\D} = \lpr{\vv{u}}{\D} = \vv{u}$ has norm $k\leq \D$, whence $s=\infty$.
On the other hand, if $p \equiv -1 \bmod \D$, then
$\norm{\lpr{\vectorscalar{p^e}{\vv{u}}}{\D}} = \norm{\lpr{-\vv{u}}{\D}} = \norm{\scalarvector{\D}{\one}-\vv{u}}=n\D-k$ if $e$ is odd, and
$\norm{\lpr{\vectorscalar{p^e}{\vv{u}}}{\D}} = \norm{ \vv{u} } = k \leq \D$ if $e$ is even.
In the extreme case that $n=2$ and $k=\D$, we have $n\D-k = \D$, so $s=\infty$ and $\crit(\ideald, \vv{u}) = 1$, agreeing with our formula.
Otherwise, $\norm{\lpr{\vectorscalar{p}{\vv{u}}}{\D}}> \D$, so that $s=1$, and the formula follows.
\end{remark}

\section{An arbitrary diagonal ideal}  

We now turn to the task of determining the critical exponents and Frobenius powers of the general diagonal ideal 
  \[
  \ideald = \diag(\vv{\D}) = \ideal{x_1^{d_1}, \ldots, x_n^{d_n}},
  \]
where $\vv{\D} = (d_1, \ldots, d_n)$ is some fixed point in $\NN^n$ with $\vv{\D} > \vv{0}$.  

\subsection{Preliminaries}  

Let $d$ be the least common multiple of the entries of~$\vv{\D}$.
Consider the $\NN$-grading on the ambient polynomial ring given by \[\deg(x_i) =  \frac{\D}{d_i}\] for every $1 \leq i \leq n$.  
We extend the degree function to $\NN^n$ via the rule 
\[ \deg(\vv{u}) = \deg(x^{\vv{u}}).\] 
Observe that under this grading, $\ideald$ is generated by monomials of degree $\D$.

\begin{definition}
   Given  $\vv{u} \in \NN^n$,  $\bar{\vv{u}}$ is the point in $\NN^n$ given by
   \[ \bar{\vv{u}} = \bigg( \frac{\D u_1}{d_1}, \ldots, \frac{\D u_n}{d_n} \bigg). \]  
\end{definition}

It follows immediately from this definition that $\deg(\vv{u}) = \norm{\bar{\vv{u}}}$. 

\subsection{Critical exponents.}

\begin{lemma}
\label{balancing crits: L}
Let $\bar{\ideald} = \diag(\bar{\vv{\D}}) = \ideal{ x_1^\D, \ldots, x_n^\D}$.
If $\vv{u} \in \NN^n$ is a point with positive coordinates, then $\crit(\ideald, \vv{u}) = \crit(\bar{\ideald}, \bar{\vv{u}})$.
\end{lemma}

\begin{proof}
   If $m$ is a positive integer, then there is bijection between the minimal monomial generators of $\frob{\ideald}{m}$ and $\frob{\bar{\ideald}}{m}$ given by 
   \[ x_1^{d_1s_1} \cdots x_n^{d_n s_n}  \in \frob{\ideald}{m} \longleftrightarrow  x_1^{\D s_1} \cdots x_n^{\D s_n} \in \frob{\bar{\ideald}}{m} \]
   As the reader can immediately verify, the definition of $\bar{\vv{u}}$ is such that a minimal monomial generator of $\frob{\ideald}{m}$ lies in $\diag(\vectorscalar{q}{\vv{u}})$ if and only if the corresponding generator of  $\frob{\bar{\ideald}}{m}$ lies in $\diag(\vectorscalar{q}{\bar{\vv{u}}})$.
   In particular, $\mu(\ideald, \vv{u}, q)$ equals $\mu(\bar{\ideald}, \bar{\vv{u}}, q)$ for every $q$, and our claim follows.
 \end{proof}

The following is a direct analog of \Cref{critical exponents of diagonal can be taken with respect to diagonal ideals: P}.

\begin{proposition}
   \label{critical exponents of unbalanced diagonal can be taken with respect to diagonal ideals: P}
   Every critical exponent of $\ideald =  \ideal{x_1^{d_1}, \ldots, x_n^{d_n}}$ in the unit interval has the form $\crit(\ideald, \vv{u})$, where the point $\vv{u}\in \NN^n$ satisfies $\vv{0} < \vv{u} \leq \vv{\D}$.
   In addition, if $\vv{u} \not < \vv{\D}$, then $\crit(\ideald, \vv{u}) = 1$.
   \qed
\end{proposition}

\begin{convention} We continue to use $d$ to denote the least common multiple of the coordinates of $\vv{d}$.  
As in \Cref{diagonalsBalanced: S}, for the remainder of this section, the expression $p \gg 0$ means that $p > n\D-n$ and $p \nmid \D$.  
\end{convention}

\begin{theorem}  
\label{crits of general diagonal ideal: T}
Suppose $p \gg 0$, and consider a point $\vv{u} \in \NN^n$ with $\vv{0} < \vv{u} \leq \vv{\D}$.   
\begin{enumerate}  
\item If $\deg(\vv{u}) > \D$, then $\crit(\ideald, \vv{u}) = 1$.
\item Otherwise, if $s = \inf \{ e \geq 1 : \norm{\lpr{\vectorscalar{p^e}{\bar{\vv{u}}}}{\D}} > \D \}$, then
 \[ \crit(\ideald, \vv{u}) = \frac{ \deg(\vv{u})}{\D} - \frac{ \norm{\lpr{\vectorscalar{p^s}{\bar{\vv{u}}}}{\D}} - \D}{\D p^s}, \]
which we interpret to agree with $\deg(\vv{u})/\D$ if $s = \infty$.
\end{enumerate}
\end{theorem}

\begin{proof}  Given that $\norm{\bar{\vv{u}}} = \deg(\vv{u})$, the first assertion follows from \Cref{balancing crits: L} and \Cref{degenerate diagonal crit: P}, and the second from \Cref{balancing crits: L} and \Cref{critical exponents of diagonal ideal: T}.
\end{proof}

\subsection{Frobenius powers}

Here, we determine the Frobenius powers $\frob{\ideald}{t}$.
Once again, by Skoda's Theorem for Frobenius powers, it suffices to assume that $t$ is a critical exponent of $\ideald$ in the open unit interval.  

\begin{setup}
   \label{unbalanced diag: setup}
Suppose that $p \gg 0$.  
\begin{enumerate}
\item $\vv{u} \in \NN^n$ is a point satisfying $\vv{u} > \vv{0}$ and $\crit(\ideald, \vv{u}) < 1$.
\item For each $m\in \NN$, we set $R_{>m} = \ideal{x^{\vv{v}} : \deg(\vv{v}) > m}$.
In particular, 
\[
R_{>\deg(\vv{u}-\one)} = \ideal{x^{\vv{v}} : \deg(\vv{v} + \one) > \deg(\vv{u})}.
\]
\item $\mathcal{A}$ is the set consisting of all points $\vv{v} \in \NN^n$ satisfying $\crit(\ideald, \vv{v}+\one) > \crit(\ideald, \vv{u})$ and $\deg(\vv{v} + \one) = \deg(\vv{u})$, and $\ideala$ is the ideal generated by all monomials $x^{\vv{v}}$ with $\vv{v} \in \mathcal{A}$.
\end{enumerate}
\end{setup}

\begin{remark}  
Once $\vv{u}$ is specified, the generators of $R_{>\deg(\vv{u}-\one)}$ depend on $\deg(\vv{u})$, but are independent of $p$, so $R_{>\deg(\vv{u}-\one)}$ can be effectively computed.
On the other hand, $\mathcal{A}$ depends both on $\vv{u}$ and on $p$ modulo $\D$.  
Nevertheless, the condition that $\deg(\vv{v} + \one) = \deg(\vv{u})$ is satisfied by only finitely many points $\vv{v} \in \NN^n$, and for such point $\vv{v}$,  the relationship between $\crit(\ideald, \vv{v}+\one)$ and $\crit(\ideald, \vv{u})$ depends only on $p$ modulo $\D$ (see \Cref{comparing critical exponents of balanced diagonal: remark}).
In particular, $\mathcal{A}$ (and consequently $\ideala$) can be effectively computed once $\vv{u}$ is specified, and takes on only finitely many values as $p$ varies. 
\end{remark}

\begin{theorem} 
\label{Frobenius powers of general diagonal ideal: T}
 In the context of \Cref{unbalanced diag: setup}, if $p \gg 0$ and $\lambda = \crit(\ideald, \vv{u})$, then \[\frob{\ideald}{\lambda} = R_{>\deg(\vv{u}-\one)} + \ideala.\]
\end{theorem}

\begin{proof}
   \Cref{crits of mon ideals: P} tells us that $x^\vv{v}\in \frob{\ideald}{\lambda}$ if and only if $\crit(\ideald,\vv{v}+\one)>\lambda=\crit(\ideald,\vv{u})$.
   By \Cref{crits of general diagonal ideal: T} (\cf\ \Cref{comparing critical exponents of balanced diagonal: remark}), the latter condition holds if and only if $\deg(\vv{v}+\one)> \deg(\vv{u})$, meaning that $x^\vv{v}\in R_{>\deg(\vv{u}-\one)}$, or $\deg(\vv{v}+\one)=  \deg(\vv{u})$ and $\crit(\ideald,\vv{v}+\one)>\crit(\ideald,\vv{u})$, meaning that $\vv{v}\in \mathcal{A}$.
\end{proof}

Below, we present a concise description of the behavior of the Frobenius powers of $\ideald$ in an important special case.

\begin{corollary}
\label{diagonal and p=1: C}
If $p \gg 0$ and $p \equiv 1 \bmod \D$, then the only critical exponents of $\ideald$ in the open unit interval are the rational numbers $k/\D$ for some integer $n \leq k < \D$, and \[ \frob{\ideald}{k/\D} = R_{>k-n}.\]
\end{corollary}

\begin{proof}
The assertion regarding the critical exponents of $\ideald$ follows from \Cref{critical exponents of unbalanced diagonal can be taken with respect to diagonal ideals: P}, \Cref{balancing crits: L}, and \Cref{balanced diagonal special congruences: remark}, and the formula for the Frobenius powers at the parameters $k/\D$ follows from this description of the critical exponents of $\ideald$ and \Cref{Frobenius powers of general diagonal ideal: T}.
\end{proof}

In the balanced case considered in the previous section, where $d_1= \cdots = d_n = \D$, the grading on the ambient polynomial ring is the standard one.
In other words, $\deg(\vv{u}) = \norm{\vv{u}}$ for every $\vv{u} \in \NN^n$, and so \Cref{Frobenius powers of general diagonal ideal: T} takes the following simple form. 

\begin{theorem}  
\label{Frobenius powers of balanced diagonal ideal: T}
Suppose that $\ideald = \ideal{x_1^\D, \ldots, x_n^\D}$ and $p \gg 0$. If $k$ is an integer with $n \leq k \leq \D$, and $\vv{u} \in \comp(k,n)$ is a point such that $\lambda = \crit(\ideald, \vv{u})$ lies in the open unit interval, then 
 \[ 
 \pushQED{\qed}
 \frob{\ideald}{\lambda} = \idealm^{k-n+1} + \ideal{x^{\vv{v}} :  \norm{\vv{v}} = k-n \text{ and } \crit(\ideald, \vv{v}+\one) > \lambda}. \qedhere\popQED\]
\end{theorem}

\subsection{Examples}
\label{diagonal examples: SS}

\begin{example}
\label{diagonal simple 1: E}
Consider the balanced diagonal ideal $\ideald = \ideal{x^7, y^7, z^7}$, contained in $\idealm = \ideal{x, y, z}$.
If $p \gg 0$ and $p \equiv 6 \bmod 7$, then we have
\begin{equation*}
   \frob{\ideald}{t}=
   \begin{cases}
      \kk[x,y,z] &\text{if } t \in \big[0, \frac{3}{7} - \frac{11}{7p} \big) \\
      \idealm &\text{if } t \in \big[\frac{3}{7} - \frac{11}{7p}, \frac{4}{7} - \frac{10}{7p} \big) \\
      \idealm^2 &\text{if } t \in \big[\frac{4}{7} - \frac{10}{7p} , \frac{5}{7} - \frac{9}{7p} \big) \\
      \idealm^3 &\text{if } t \in \big[ \frac{5}{7} - \frac{9}{7p} , \frac{6}{7} - \frac{8}{7p}\big) \\
      \idealm^4 &\text{if } t \in \big[ \frac{6}{7} - \frac{8}{7p}, 1- \frac{1}{p} \big) \\
      \idealm^5 &\text{if } t \in \big[ 1 - \frac{1}{p} , 1 \big)
   \end{cases}
\end{equation*}
The absence of the second monomial summand from \Cref{Frobenius powers of balanced diagonal ideal: T} reflects the fact that in this setting, the critical exponents $\crit(\ideald, \vv{u})$ associated to points $\vv{u} \in \comp(k,n)$ are independent of $\vv{u}$ for every $3 \leq k \leq 7$.
This observation is a special case of the computation appearing in \Cref{balanced diagonal special congruences: remark}.

The reader should also contrast this with \Cref{power of m simple 1: E}---for instance, we saw there that $\frob{(\idealm^7)}{t}\neq \idealm^5$ for every parameter $0 \leq t < 1$.
\end{example}		
		
\begin{example}  Let $\ideald$ and $\idealm$ be as in \Cref{diagonal simple 1: E}.
If we instead suppose that $p \gg 0$ and $p \equiv 5 \bmod 7$, then 
		\begin{equation*}
			\frob{\ideald}{t}=
			\begin{cases} 
				\kk[x,y,z] &\text{if } t \in \big[0, \frac{3}{7} - \frac{8}{7p} \big) \\ 		
				\idealm &\text{if } t \in \big[\frac{3}{7} - \frac{8}{7p}, \frac{4}{7} - \frac{6}{7p} \big) \\ 	
				\idealm^2 &\text{if } t \in \big[\frac{4}{7} - \frac{6}{7p} , \frac{5}{7} - \frac{4}{7p} \big) \\ 		
				\idealm^3 &\text{if } t \in \big[ \frac{5}{7} - \frac{4}{7p} , \frac{6}{7} - \frac{9}{7p}\big) \\ 
				\idealm^4 + \ideal{xyz, x^2y, x^2z, xy^2,xz^2, y^2z, yz^2} &\text{if } t \in \big[ \frac{6}{7} - \frac{9}{7p}, \frac{6}{7} - \frac{2}{7p} \big) \\ 
				\idealm^4 &\text{if } t \in \big[ \frac{6}{7} - \frac{2}{7p}, 1- \frac{1}{p} \big) \\ 
				\idealm^5 + \ideal{x^2 y z, x y^2 z, x y z^2, x^2 y^2, x^2 z^2, y^2 z^2} &\text{if } t \in \big[ 1 - \frac{1}{p}, 1- \frac{1}{p^2} \big) \\ 
				\idealm^5 + \ideal{x^2 y z, x y^2 z, x y z^2} &\text{if } t \in \big[ 1 - \frac{1}{p^2}, 1- \frac{1}{p^3} \big) \\ 
				\idealm^5 &\text{if } t \in \big[ 1 - \frac{1}{p^3} , 1 \big)
		\end{cases}
		\end{equation*}
	As in \Cref{diagonal simple 1: E}, the absence of a second monomial summand in the first handful of formulas is due to the fact that $\crit(\ideald, \vv{u})$ is independent of the point $\vv{u} \in \comp(k,n)$ for every $3 \leq k \leq 5$.    
	For $k=3$ and $k=4$, this is explained by observing that there is a unique point in $\comp(k,n)$, up to permutation of coordinates. 
	When $k=5$, the set $\comp(k,n)$ consists of $(3,1,1)$ and $(2,2,1)$ and permutations, and these points happen to give rise to the same critical exponent for our choice of $p$. 
	
	On the other hand, when $k=6$ and $k=7$,  $\crit(\ideald, \vv{u})$ takes on multiple values (two when $k=6$, and three when $k=7)$ as $\vv{u} \in \comp(k,n)$ varies.
\end{example} 

In the following two examples, we see that for balanced diagonal ideals, the number of critical exponents with respect to points in $\comp(k,n)$, for some fixed $n \leq k \leq \D$, can be quite large, especially when $k$ is close to $\D$.    
Note that once $k$ is fixed, this number of critical exponents is equal to the number of distinct Frobenius powers $\frob{\ideald}{t}$ with $t\in\big(\frac{k-1}\D, \frac k\D\big]$, for $p\gg 0$.  

\begin{example} Suppose $\ideald = \ideal{x^{35}, y^{35}, z^{35}, w^{35}}$. If $p \gg 0$ and $p \equiv 3 \bmod 35$, then $\ideald$ has ten critical exponents in the interval $\left( \frac{34}{35}, 1 \right)$,
namely, 
\[
1 - \frac{i}{p^j}, \text{ for } 1 \leq i \leq 2  \text{ and } 1 \leq j \leq 5. 
\]
\end{example}

\begin{example}
Suppose $\ideald = \ideal{x^{47}, y^{47}}$.  If $p \gg 0$ and $p \equiv 7 \bmod 47$, then 
the ten critical exponents of $\ideald$ in the interval $\left( \frac{39}{47}, \frac{40}{47} \right)$ are
\begin{alignat*}{5}
&\frac{40}{47} - \frac{45}{47 p} & \hspace{0.8cm} & \frac{40}{47} - \frac{33}{47 p^2} & \hspace{0.8cm} & \frac{40}{47} - \frac{43}{47 p^3} & \hspace{0.8cm} & \frac{40}{47} - \frac{19}{47 p^4} & \hspace{0.8cm} & \frac{40}{47} - \frac{39}{47 p^5} \\
&\frac{40}{47} - \frac{31}{47 p^7} & \hspace{0.8cm} & \frac{40}{47} - \frac{29}{47 p^8} & \hspace{0.8cm} & \frac{40}{47} - \frac{15}{47 p^9} & \hspace{0.8cm} & \frac{40}{47} - \frac{11}{47 p^{10}} & \hspace{0.8cm} & \frac{40}{47} - \frac{13}{47 p^{13}}
\end{alignat*}
\end{example}

We conclude this section with an example in the unbalanced setting.

\begin{example}
Fix ideals $\ideald = \ideal{x^6, y^4} \subseteq \idealm = \ideal{x, y}$ of $R = \kk[x,y]$. 
\begin{figure}
\begin{center}  
    \includegraphics[width=0.5\textwidth]{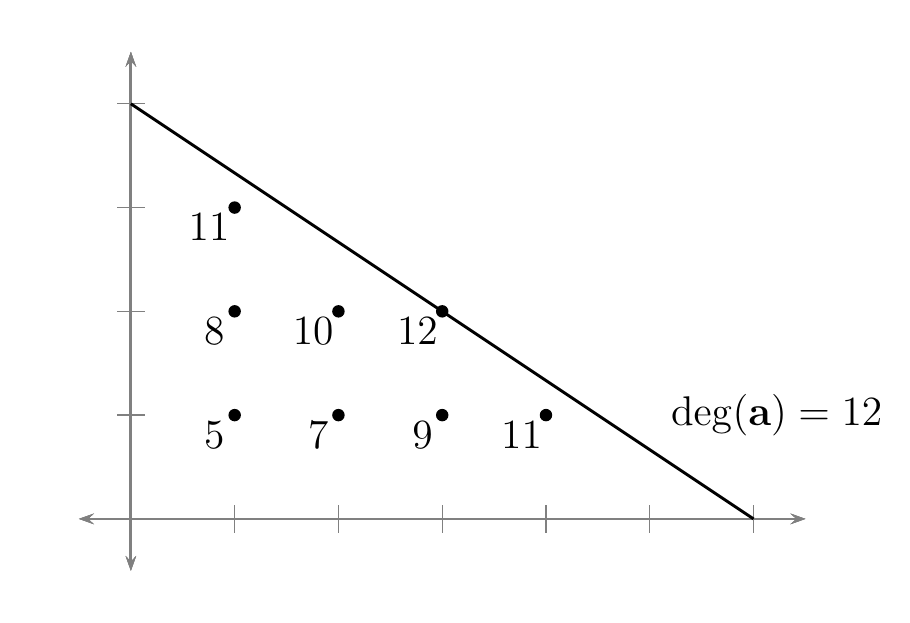}
\end{center}
\vspace{-0.8cm}
\caption{ \label{unbalanced-example: figure}  All $\vv{u} \in \NN^2$ with $\vv{u} > \vv{0}$ and $\deg(\vv{u}) \leq 12$.} 
\end{figure}
Here, $\deg(x)=2$ and $\deg(y)=3$, and the set of all points $\vv{u} > \vv{0}$ in $\NN^2$ for which $\deg(\vv{u}) \leq 12$ determines all critical exponents of $\ideald$ in the unit interval; \Cref{unbalanced-example: figure} illustrates these points, and their degrees. 
When $p \gg 0$ and $p \equiv 5 \bmod 12$, the points
 $(1,3)$ and $(4,1)$, both of degree $11$, determine distinct critical exponents, and the critical exponent with respect to $(3,2)$ is $1$.  Moreover,
		\begin{equation*}
			\frob{\ideald}{t}=
			\begin{cases} 
				R &\text{if } t \in \big[0, \frac{5}{12} - \frac{1}{12p} \big) \\ 		
				\idealm &\text{if } t \in \big[\frac{5}{12} - \frac{1}{12p}, \frac{7}{12}  \big) \\ 	
				R_{> 2} &\text{if } t \in \big[\frac{7}{12}, \frac{2}{3} - \frac{1}{3p} \big) \\ 		
				R_{> 3} &\text{if } t \in \big[ \frac{2}{3} - \frac{1}{3p} , \frac{3}{4} \big) \\ 
				R_{> 4} &\text{if } t \in \big[ \frac{3}{4}, \frac{5}{6} - \frac{1}{6p} \big) \\ 
				R_{> 5} &\text{if } t \in \big[ \frac{5}{6}-\frac{1}{6p}, \frac{11}{12}- \frac{7}{12p} \big) \\ 
				R_{> 6} + \ideal{x^3}  &\text{if } t \in \big[ \frac{11}{12}- \frac{7}{12p}, \frac{11}{12} \big) \\ 
				R_{> 6} &\text{if } t \in \big[ \frac{11}{12}, 1 \big) \\ 
		\end{cases}
		\end{equation*}
\end{example}

\section{Applications to generic hypersurfaces}

In this section, we leverage our results on the Frobenius powers and critical exponents of the monomial ideals $\idealm^\D$ and $\ideald$ to study the test ideals and $F$-jumping exponents of generic elements of these ideals, that is, various $\kk^{\ast}$-linear combinations of their monomial generators.

We start by considering generic elements of $\idealm^\D$.
The first statement below realizes the test ideals of a \emph{very general} homogeneous polynomial of degree $\D$ as the Frobenius powers of its term ideal,  which can then be explicitly computed using \Cref{Frobenius powers of m^d: T}, as in \Cref{m^d examples: SS}.
This conclusion is significantly stronger than that in the second statement, which instead concerns a subset of the $F$-jumping exponents of a \emph{general} homogeneous polynomial.

\begin{corollary}
\label{generic homogeneous poly: C}
Suppose $p>\D$.
Let $x^{\vv{a}_1}, \ldots, x^{\vv{a}_l}$ be the distinct monomials of degree $\D$ in the variables $x_1, \ldots, x_n$, and consider the polynomial 
 \[ f = \alpha_1  x^{\vv{a}_1} + \cdots + \alpha_l x^{\vv{a}_l} \in \idealm^\D \]
where $\alpha=(\alpha_1, \ldots, \alpha_l) \in \kk^l$.
\begin{enumerate}
\item If $\kk=\FF_p(\alpha_{1},\ldots,\alpha_{l})$, and $\alpha_{1},\ldots,\alpha_{l}$ are algebraically independent over $\FF_p$, then $\tau(f^{t}) =  \frob{(\idealm^\D)}{t}$ for every $0 < t < 1$.
\item There exists a finite set of nonzero polynomials $\Sigma \subseteq \FF_p[z_1, \ldots, z_l]$, which may depend on $n$ and $\D$, but is independent of the field $\kk$, with the following property\textup:  If $\alpha$ is not in the closed set $\VV(\Sigma)$ of $\kk^l$, then every critical exponent of $\idealm^\D$ is an $F$-jumping exponent of $f$.  
\end{enumerate}
\end{corollary}

\begin{proof}
   In the first case, we have that $\tau(f^{t}) =  \frob{(\idealn^\D)}{t} R = \frob{(\idealm^\D)}{t}$, where $\idealn = \ideal{x_1, \ldots, x_n}  \subseteq \FF_p[x_1, \ldots, x_n]$.
   Indeed, the first equality follows from  \cite[Corollary~5.7]{hernandez+etal.frobenius_powers}, and the second from the fact that $\idealn$ is a monomial ideal and $\idealm = \idealn R$.

For the second statement, let $\lambda$ be a critical exponent of $\idealm^\D$.  
By the versions of Skoda's Theorem for Frobenius powers and test ideals (\cite[Corollary~3.17]{hernandez+etal.frobenius_powers}, \cite[Proposition~2.25]{blickle+mustata+smith.discr_rat_FPTs}), we may assume that $0<\lambda<1$. 
 \Cref{small crits of m^D: P,degenerate m^D: P} tell us that $\lambda = \crit(\idealm^\D, \vv{u})$, where $\vv{u} \in \NN^n$ satisfies $\vv{u} > \vv{0}$ and $\norm{\vv{u}}< \D$.
It follows from \cite[Theorem~5.16]{hernandez+etal.frobenius_powers} that for each $\vv{u}$, there exists a set $\Sigma_{\vv{u}} \subseteq \FF_p[z_1, \ldots, z_l]$ of nonzero polynomials such that $\lambda = \crit(f, \vv{u})$ whenever $\alpha \notin \VV(\Sigma_{\vv{u}})$.  
We can then take $\Sigma$ to be the union of all $\Sigma_{\vv{u}}$, where $\vv{u} \in \NN^n$ varies through the finitely many points satisfying the above conditions. 
\end{proof}

We now turn our attention to {arbitrary} diagonal polynomials.  Unlike \Cref{generic homogeneous poly: C}, the following result does not follow directly from ones in \cite{hernandez+etal.frobenius_powers}, but instead from a computation of test ideals.  In its proof, we will need to refer to a well-known result from \cite{blickle+mustata+smith.F-thresholds_hyper} describing test ideals of hypersurfaces at special parameters.  

\begin{proposition}  
\label{diagonal poly: P}
Consider $g = \beta_1 x_1^{d_1} + \cdots + \beta_n x_n^{d_n} \in \ideald=\ideal{x_1^{d_1},\ldots,x_n^{d_n}}$, where each $\beta_i \in \kk^{\ast}$.   
If $p$ does not divide any $d_i$, then $\tau(g^{t}) = \frob{\ideald}{t}$ for every $0 < t < 1$.
\end{proposition}

\begin{proof}
We can assume that $t = m/q$, for $q$ a power of $p$, and $m$ a positive integer less than $q$.
Then $\tau(g^{t}) = \frob{\ideal{g^m}}{1/q}$ by  \cite[Lemma~2.1]{blickle+mustata+smith.F-thresholds_hyper}.  
If $D$ is the diagonal matrix whose $i$-th diagonal entry is $d_i$, then $g^m = \sum_{\vv{u} \in \Omega} \binom{m}{\vv{u}} \beta^{\vv{u}} x^{D \vv{u}}$, where $\Omega$ consists of all  points $\vv{u} \in \NN^n$ with  $\binom{m}{\vv{u}} \not\equiv 0 \bmod p$.
We claim that if $\vv{u}, \vv{v} \in \Omega$ are distinct,  then the remainders when dividing $D \vv{u}$ and $D\vv{v}$ by $q$ are distinct.
This is equivalent to the assertion that the monomials $x^{D \vv{u}}$ and $x^{D \vv{v}}$ are $R^q$-multiples of distinct elements of the monomial basis for $R$ over $R^q$.
Granting this, it then follows from the description of Frobenius $q$-th roots provided in \cite[Proposition~2.5]{blickle+mustata+smith.discr_rat_FPTs} that  
\[ \tau(g^{t}) = \frob{\ideal{g^m}}{1/q} = \frob{\ideal{ x^{D\vv{u}} : \vv{u} \in \Omega}}{1/q} = \frob{\ideald}{t},\] 
as desired.

It remains to justify the claim.
We verify the contrapositive:  $D \vv{u}$ and $D \vv{v}$ have the same remainder when dividing by $q$ if and only if $d_i u_i \equiv d_i v_i \bmod q$ for every $i$. As $d_i$ is a unit modulo $q$, this is equivalent to $u_i \equiv v_i \bmod q$ for all $i$.  From this, and the bounds  $0 \leq u_i, v_i  \leq \norm{\vv{u}} = \norm{\vv{v}} = m < q$, we finally conclude that $\vv{u} = \vv{v}$.  
\end{proof}

\begin{remark}
The above proof is essentially the same as that of \cite[Lemma~4.7]{hernandez.diag_hypersurf}, where it is proven that $\tau(g^{t})$ is a monomial ideal for each $0<t<1$.
In fact, using \cite[Lemma~2.1]{blickle+mustata+smith.F-thresholds_hyper} and \cite[Proposition~2.5]{blickle+mustata+smith.discr_rat_FPTs} one can see that, more generally, whenever a test ideal $\tau(f^{t})$ of a polynomial $f$ is a monomial ideal, it agrees with the $t$-th Frobenius power of the term ideal of $f$.
\end{remark}

\subsection{Reduction to prime characteristic}  

Recall that if $\ideala$ is an ideal of a polynomial ring over a field of characteristic zero, then its multiplier ideals $\mathcal{J}(\ideala^{t})$ are  ideals of the ambient polynomial ring of $\ideala$ that are indexed by a nonnegative real parameter $t$, which can be defined in terms of resolutions of singularities, or through more analytic means.
For more on this topic, we refer the reader to the concrete survey \cite{blickle+lazarsfeld.intro_multiplier_ideals}, but point out that the general theory of multiplier ideals will \emph{not} be called upon here.  

At this point, it is well known that the multiplier ideals associated to an ideal $\ideala$ as above may be regarded as the limiting values of the test ideals of the reductions of $\ideala$ to characteristic $p>0$, as $p$ tends to infinity (see, \eg \cite{smith.mult-ideal-is-univ-test-ideal, hara+yoshida.generalization_TC_multiplier_ideals}), and it is conjectured that these ideals agree at all parameters for infinitely many primes (see, \eg \cite[Conjecture~1.2]{mustata+srinivas.ordinary_varieties}).
In \Cref{test=multiplier for diagonal: T} below, we verify this conjecture for diagonal hypersurfaces.
To simplify the formalism of reduction to prime characteristic, we restrict our attention to diagonal hypersurfaces with rational coefficients;  the interested reader should be able to adapt our proof to the case of a diagonal hypersurface with coefficients in an arbitrary field of characteristic zero without too much effort.
In the statement of \Cref{test=multiplier for diagonal: T}, we use the subscript ``$p$'' to indicate reduction modulo $p$. 

\begin{theorem}
\label{test=multiplier for diagonal: T}
If $g$ is a diagonal polynomial over $\QQ$, then there exist infinitely many primes $p$ such that $ \tau(g_p^{t}) = \mathcal{J}(g^{t})_p$ for every $t \geq 0$.  
\end{theorem}

\begin{proof}  
The versions of Skoda's Theorem for multiplier and test ideals allow us to assume that $t \in (0,1)$, and it is well known that the multiplier ideals of $g$ (when computed over $\QQ$, or any field of characteristic zero, for that matter) are monomial ideals for such parameters.  In fact, $\mathcal{J}(g^t) = \mathcal{J}(\ideald^t)$ for every parameter $t \in (0,1)$, where $\ideald$ is the term ideal of $g$ \cite[Example~9, Corollary~13]{howald.multiplier_ideals_of_generic_polynomials}.  In particular, the reductions $\mathcal{J}(g^t)_p = \mathcal{J}(\ideald^t)_p$ are monomial ideals over $\FF_p$ for every prime $p$.

On the other hand, the reductions $g_p$ for $p \gg 0$ are obtained by regarding the fractions appearing in $g$ as elements of $\FF_p$.  In particular, the reduction $\ideald_p$ of the term ideal $\ideald$ of $g$ is the term ideal of $g_p$ when $p \gg 0$, and it then follows from \Cref{diagonal poly: P} that $\tau(g_p^t) = \frob{\ideald_p}{t}$ when $p \gg 0$ and $t \in (0,1)$.  

Given these observations, it suffices to show that \[ \frob{\ideald_p}{t} = \mathcal{J}(\ideald^t)_p \] for every $t \in (0,1)$ whenever $p \gg 0$ and $p \equiv 1 \bmod \D$, where $\D$ is the least common multiple of the exponents appearing in $\ideald$.  However, in this context, the ideals $\frob{\ideald_p}{t}$ are completely described by \Cref{diagonal and p=1: C}, and it is straightforward to verify that this agrees with the description of the monomial ideal $\mathcal{J}(\ideald^t)$ in terms of the Newton polyhedron of $\ideald$ provided in \cite{howald.multiplier_ideals_of_monomial_ideals}.
\end{proof}

\begin{remark}\label{uniform behavior: R}
	A key observation in the proof of \Cref{test=multiplier for diagonal: T} was that \[ \tau(g_p^t) = \frob{\ideald_p}{t} \text{ when $p \gg 0$ and $t \in (0,1)$. } \]
	Given this, one may use \Cref{Frobenius powers of general diagonal ideal: T} to explicitly compute $\tau(g_p^t)$ when $p \gg 0$,  as in \Cref{diagonal examples: SS}, which, in particular,   
	 shows that the family $(\tau(g_p^t))_t$ depends on the class of $p$ modulo $\D$ in a uniform way.  
	Though we will not pursue this further in this article, we point out that this demonstrates that every diagonal hypersurface satisfies the conditions appearing in \cite[Problems~3.7, 3.8, and~3.10]{mustata+takagi+watanabe.F-thresholds}.
\end{remark}

\begin{remark}  
	Though, as noted above, the property of being a generic element of a monomial ideal is preserved under reduction to characteristic $p  \gg 0$, the property of being a \emph{very general} generic element is not.  
	Consequently, one cannot hope to apply the first part of \Cref{generic homogeneous poly: C} to obtain a statement analogous to \Cref{test=multiplier for diagonal: T} for a generic homogeneous polynomial.  
	In fact, the family $(\tau(h_p^t))_t$ when $h$ is a homogeneous polynomial of degree at least two over a field of characteristic zero is known to be quite complicated; \eg see \cite{ftf1, pagi.fpt_of_bivariate_forms,hara.F-pure-thresholds,mustata+takagi+watanabe.F-thresholds}.
        The behavior exhibited in these paper suggests that such polynomials are unlikely to possess the uniform behavior described in \Cref{uniform behavior: R}.
\end{remark}

\section*{Acknowledgements}The authors would like to thank the University of Kansas for hosting part of this collaboration, and the National Science Foundation for their support;  during the production of this research, the first author was supported by NSF grant DMS-1600702, and the third by NSF grant DMS-1623035.

{\small
\bibliographystyle{amsalpha}
\bibliography{bibdatabase}
}

\end{document}